\documentclass[reqno]{amsart}
\usepackage{amsmath,amssymb,amsthm,amsrefs}
\usepackage{xcolor}

\newtheorem{teo}{Theorem}[section]
\newtheorem{prop}[teo]{Proposition}
\newtheorem{lema}[teo]{Lemma}
\newtheorem{coro}[teo]{Corolary}
\theoremstyle{definition}
\newtheorem{dfn}[teo]{Definition}

\newtheorem{obs}[teo]{Remark}

\newtheorem{ex}[teo]{Example}

\newcommand{\nd}{\noindent}
\newcommand{\vu}{\vspace{.1cm}}
\newcommand{\vd}{\vspace{.2cm}}
\newcommand{\vt}{\vspace{.3cm}}

\newcommand{\El}{\varepsilon_L}
\newcommand{\Er}{\varepsilon_R}
\newcommand{\Ve}{\varepsilon}
\newcommand{\Ref}[1]{\overset{\eqref{#1}}{=}}

\title[Partial Actions of Weak Hopf Algebras]{Partial Actions of Weak Hopf Algebras: Smash Product, Globalization and Morita Theory}

\author[Castro]{Felipe Castro}
\address[Castro]{Universidade Federal do Rio Grande do Sul, Brazil}
\email{f.castro@ufrgs.br}

\author[Paques]{Antonio Paques}
\address[Paques]{Universidade Federal do Rio Grande do Sul, Brazil}
\email{paques@mat.ufrgs.br}

\author[Quadros]{Glauber Quadros}
\address[Quadros]{Universidade Federal do Rio Grande do Sul, Brazil}
\email{glauber.quadros@ufrgs.br}

\author[Sant'Ana]{Alveri Sant'Ana}
\address[Sant'Ana]{Universidade Federal do Rio Grande do Sul, Brazil}
\email{alveri@mat.ufrgs.br}

\thanks{The first and the third authors were partially supported by CNPq, Brazil}

\begin{document}

\begin{abstract}
	In this paper we introduce the notion of partial action of a weak Hopf algebra on algebras,
unifying the notions of partial group action \cite{DE}, partial Hopf action (\cite{AB},\cite{AB1},\cite{CJ}) and
partial groupoid action \cite{BP}. We construct the fundamental tools to develop this new subject, namely,
the partial smash product and the globalization of a partial action, as well as we establish a connection
between partial and global smash products via the construction of a surjective Morita context. In particular,
in the case that the globalization is unital, these smash products are Morita equivalent.
We show that there is a bijective correspondence between globalizable partial groupoid actions and symmetric
partial groupoid algebra actions, extending similar result for group actions \cite{CJ}.
Moreover, as an application we give a complete description of all partial
actions of a weak Hopf algebra on its ground field, which suggests a method to construct more general
examples.
\end{abstract}

\maketitle



%
%

\section{Introduction}

Partial actions of groups on algebras was introduced in the literature by R. Exel in \cite{E}.
His main purpose in that paper was to develop a method that allowed to describe the
structure of $C^*$-algebras under actions of the circle group. The first approach of partial group actions
on algebras, in a purely algebraic context, appears later in a paper by M. Dokuchaev and R. Exel \cite{DE}.

Partial group actions can be easily obtained by restriction from the global ones, and this fact stimulated the
interest on knowing under what conditions (if any) a given partial group action is of this type. In the
topological context this question was dealt with by F. Abadie in \cite{abadie}. The algebraic version of a
globalization (or enveloping action) of a partial group action, as well as the study about its existence,
was also considered by M. Dokuchaev and R. Exel in \cite{DE}. A nice approach on the relevance of the
relationship between partial and global group actions, in several branches of mathematics,
can be seen in \cite{D}.

\vu

As a natural task, S. Caenepeel and K. Janssen \cite{CJ} extended the  notion of partial group action to the
setting of Hopf algebras and developed a theory of partial (co)actions of Hopf algebras, as well as
a partial Hopf-Galois theory. Based on the Caenepeel-Janssen's work, E. Batista and M. Alves in \cites{AB, AB1}
showed that every partial action of a Hopf algebra has a globalization and that the
corresponding partial and global smash products are related by a surjective Morita context. At almost the same time
D. Bagio and A. Paques developed a theory of partial groupoid actions extending, in particular, results of M.
Dokuchaev and R. Exel in \cite{DE} about partial group actions and their globalizations.

\vu

Actions of weak Hopf algebras is the precise context to unify all these above mentioned theories, and this is our main purpose.

\vu

In this paper we deal with actions of weak Hopf algebras and extend to this setting many of the results
above mentioned. As it is well known, Hopf algebras and groupoid algebras are perhaps the simplest examples
of weak Hopf algebras. The weak Hopf algebra theory has been started at the end of the 90's by G. B\"ohm, F. Nill and
K. Szlach\'anyi \cites{BS, BNS, BNS1, N, S}.

\vu

One of our main goals is to show that the notion of globalization can be extended to partial module algebras over weak Hopf algebras.
We succeed to prove that every (left) partial module algebra over weak Hopf algebras has a globalization
(also called enveloping action), extending the corresponding results on partial Hopf actions obtained by E. Batista e M. Alves
in \cite{AB} and \cite{AB1}. We also prove
the existence of minimal globalizations and that any two of them are equivalent, as well as that any
globalization is a homomorphic preimage of a minimal one (see Section \ref{glob}).

\vu

The other one is to ensure that the partial version of the smash product, as introduced by D. Nikshych in \cite{Nik},
can also be obtained (see Section \ref{smash}). The hardest task here is to show that such a partial smash
product is well defined, just because the tensor product, in this case, is not over the ground field.
The usual theory does not fit into our context since the definitions for partial structures are a bit different.
 The existence of partial smash products allows us to construct a surjective Morita context relating them with the
corresponding global ones (see Section \ref{morita}).

As an application, we describe completely all the partial actions of a weak Hopf algebra on its ground field,
which also suggests the construction of other examples of such partial actions, different from the canonical ones (see Section \ref{field}).
We also analyze the relation between partial groupoid actions, as introduced in \cite{BP}, and  partial
actions of groupoid algebras, showing how partial group actions, in particular, and partial groupoid actions, in general,
fit into this new context (see Section \ref{groupoid}).

\vu

In Section \ref{wpaction} we present the definitions and basic results that we will need in the sequel.

\vu

Throughout, $\Bbbk$ will denote a field and every $\Bbbk$-algebra is assumed to be associative and
unital, unless otherwise stated. Unadorned $\otimes$ means $\otimes_{\Bbbk}$.

\vu

%
%

\section{Partial actions of weak Hopf algebras}\label{wpaction}

\subsection{Weak Hopf algebras}

We start recalling the definition and some of the pro\-per\-ti\-es of a weak Hopf algebra over a field $\Bbbk$.
For more about it we refer to \cite{BNS}.

\begin{dfn}
A sixtuple $(H,m,u,\Delta,\varepsilon,S)$ is a weak Hopf algebra, with antipode $S$, if:

\vu

\begin{itemize}
	\item[(i)]  $(H,\ m,\ u)$  is a $\Bbbk$-algebra,
\vu

	\item[(ii)] $(H,\ \Delta,\ \varepsilon)$ is a $\Bbbk$-coalgebra,

\vu
	\item[(iii)] $\Delta(k h)=\Delta(k)\Delta(h)$,\ for all $h, k\in H$,

\vu

	\item[(iv)] $\varepsilon(k h_1)\varepsilon(h_2 g)=\varepsilon(khg)=\varepsilon(k h_2)\varepsilon(h_1 g),$

\vu

	\item[(v)] $(1_H \!\otimes\! \Delta(1_H))(\Delta(1_H)\!\otimes\! 1_H)=\Delta^2(1_H)=
(\Delta(1_H)\!\otimes\! 1_H)(1_H \!\otimes\! \Delta(1_H)),$

\vu

	\item[(vi)] $h_1S(h_2)=\varepsilon_{{}_{L}}(h),$

\vu

	\item[(vii)] $S(h_1)h_2=\varepsilon_{{}_{R}}(h),$

\vu
		
    \item[(viii)] $ S(h)=S(h_1)h_2S(h_3),$
\end{itemize}
\end{dfn}	
\nd where $\varepsilon_{{}_{L}}\colon H\to H$ and $\varepsilon_{{}_{R}}\colon H\to H$ are defined by
$\varepsilon_{{}_{L}}(h)=\varepsilon(1_1h)1_2$ and $\varepsilon_{{}_{R}}(h)=1_1 \varepsilon(h 1_2)$,
respectively. We will denote $H_L=\varepsilon_{{}_{L}}(H)$ and $H_R=\varepsilon_{{}_{R}}(H)$. It is clear from
these definitions that $H_L$ and $H_R$ are both finite dimensional over $\Bbbk$.

\vu

Algebras $H$ satisfying the five first statements enumerated above are simply called weak bialgebras.
As usual, we will adopt the Sweedler notation for the comultiplication $\Delta$ of $H$, that is,
$\Delta(h)=h_1\otimes h_2$ (summation understood), for any $h\in H$.

\vu

Many of the basic properties of a weak Hopf algebra proved in the finite dimensional case (see \cites{BNS,BNS1,CG})
also hold in the general case. We start by enumerating some of these properties which can be verified using
arguments similar to those used in the finite dimensional case. These properties will be very useful along this paper.

\begin{lema}\label{p1}
Let $H$ be a weak Hopf algebra. Then,
\begin{eqnarray}
\El \circ \El = \El \label{p1-1} \\
\Er \circ \Er = \Er \label{p1-2} \\
\Ve(h \El (k) ) = \Ve(hk) \label{p1-3} \\
\Ve(\Er (h) k) = \Ve(hk) \label{p1-4} \\
\El(h \El (k)) = \El(hk) \label{p1-5} \\
\Er( \Er (h) k) = \Er(hk) \label{p1-6}
\end{eqnarray}
for all $h,k$ in $H$.\qed
\end{lema}

\begin{lema}\label{p2}
Let $H$ be a weak Hopf algebra. Then, $\Delta(1_H) \in H_R \otimes H_L$.\qed
\end{lema}

\begin{lema}\label{p3}
The following statements hold:
\begin{eqnarray}
z \in H_L \Leftrightarrow \Delta(z) = 1_1 z \otimes 1_2 \label{p3-1}
\end{eqnarray}
and in this case $\Delta(z) = z 1_1 \otimes 1_2$.
\begin{eqnarray}
w \in H_R \Leftrightarrow \Delta(w) = 1_1 \otimes w 1_2 \label{p3-2}
\end{eqnarray}
and in this case $\Delta(w) = 1_1  \otimes 1_2 w$.

\vu

Moreover, $\Delta(H_L) \subset H \otimes H_L$ and $\Delta(H_R) \subset H_R \otimes H$.\qed
\end{lema}

\begin{lema} \label{p5}
$H_L$ and $H_R$ are subalgebras of $H$ with unit $1_H$. Moreover, if $z \in H_L$ and $w \in H_R$,
we have that $z w = w z$.\qed
\end{lema}

\begin{lema}\label{p4}\label{p6}\label{p7}\label{p8}
Let $H$ be a weak Hopf algebra. Then,
\begin{eqnarray}
\ h_1 \otimes h_2 S(h_3) = 1_1 h \otimes 1_2 \label{p4-1}\\
\ S(h_1)h_2 \otimes h_3 = 1_1 \otimes h 1_2 \label{p4-2}\\
h \ \El(k) = \Ve(h_1 k) h_2 \label{p4-3}\\
\Er(h) \ k = k_1 \Ve(h k_2) \label{p4-4}\\
\El(\El(h)k)  = \El(h)\El(k) \label{p6-1}\\
\Er(h\Er(k)) = \Er(h)\Er(k) \label{p6-2}\\
\El(h) = \Ve(S(h) 1_1)1_2 \label{p7-1} \\
\Er(h) = 1_1 \Ve(1_2 S(h))\label{p7-2}\\
\El(h) = S(1_1)\Ve(1_2 h) \label{p8-1}\\
\Er(h) = \Ve(1_1 h) S(1_2) \label{p8-2},
\end{eqnarray} for all $h,k \in H$.\qed
\end{lema}

\vu

\begin{lema}\label{p9}\label{p11}\label{p12}
Let $H$ be a weak Hopf algebra. Then,
\begin{eqnarray}
\El \circ S = \El \circ \Er = S \circ \Er \label{p9-1} \\
\Er \circ S = \Er \circ \El = S \circ \El \label{p9-2}\\
		S(1_1) \otimes S(1_2) = 1_2 \otimes 1_1 \label{p11-0}\\
		S( h k ) = S(k) S(h) \label{p11-1}\\
		S(h)_1 \otimes S(h)_2 = S(h_2) \otimes S(h_1) \label{p11-2}\\
		\Ve \circ S = \Ve \label{p11-3}\\
		S(1_H) = 1_H \label{p11-4}\\
		h_1 \otimes S(h_2) \ h_3 = h \ 1_1 \otimes S(1_2) \label{p12-1} \\
		h_1 \ S(h_2) \otimes h_3 = S(1_1) \otimes 1_2 \ h \label{p12-2},
	\end{eqnarray} for all $h, k\in H$. \qed
\end{lema}

The proof of the following lemma is a direct verification.

\begin{lema}\label{p10}
	Let $H$ be a weak Hopf algebra with antipode $S$. Then, $S(H_L) = H_R$ and $S(H_R) = H_L$. Moreover, $S_L=S|_{_{H_L}}$ and
$S_R=S|_{_{H_R}}$ are bijections between $H_L$ and $H_R$.\qed
\end{lema}

It is well-known that, for a weak bialgebra	$H$, $H_L$ and $H_R$ are Frobenius separable $\Bbbk$-algebras with separability idempotents
$$e_L = \El(1_1) \otimes 1_2 $$
and
$$e_R = 1_1 \otimes \Er(1_2) .$$

For a weak Hopf algebra $H$, taking into account that $\Delta(1) \in H_R \otimes H_L$, the idempotents become
$$e_L = S(1_1) \otimes 1_2 $$
and
$$e_R = 1_1 \otimes S(1_2)$$
by (\ref{p9-1}) and (\ref{p9-2}).

\vu

Then, the following result follows immediately from the central condition for separability idempotents.

\begin{lema}\label{p14}
	For all $z \in H_L$ and $w \in H_R$, we have:
	\begin{eqnarray}
	z S(1_1) \otimes 1_2 = S(1_1) \otimes 1_2 z \label{p14-1} \\
	1_1 \otimes S(1_2) w = w 1_1 \otimes S(1_2) \label{p14-2}
	\end{eqnarray}
\end{lema}

As a consequence of the above result, we obtain directly the following one.

\begin{lema}\label{p13}
	For all $z \in H_L$ and $w \in H_R$, we have:
	\begin{eqnarray}
		1_1 S_R^{-1}(z) \otimes 1_2 = 1_1 \otimes 1_2 z \label{p13-1}\\
		1_1 \otimes S_L^{-1}(w) 1_2 = w 1_1 \otimes 1_2 \label{p13-2}
	\end{eqnarray}
\end{lema}

\subsection{Partial actions}

\vu

Hereafter, all actions of a weak Hopf algebra on any algebra will be considered only on the left side.
Actions on the right side can be defined in a similar way, and corresponding results similar to the ones
we will deal with along this text can be obtained as well.
Recall that in this paper all algebras are assumed to
be associative and unital, unless otherwise stated. Furthermore, in order to avoid confusion,
we will always denote by $\cdot$ any partial action and by $\triangleright$ any global one (see, in particular, Section 5). Throughout,
$H$ will always denote a weak Hopf algebra, without any more explicit mention, unless otherwise required.
\vu

The usual definition for (global) actions of weak Hopf algebras on algebras is the following.

\begin{dfn}
Let $A$ be an algebra. A (global) \emph{action} of $H$ on $A$ is a
$\Bbbk$-linear map $\triangleright\colon H \otimes A \to A$
such that the following properties hold for all $a,b \in A$ and $h,k \in H$:

\vu

\begin{itemize}

\item[(i)] $1_H \triangleright a = a$,

\vu

\item[(ii)] $h \triangleright ab = (h_1 \triangleright a)(h_2 \triangleright b)$,

\vu

\item[(iii)] $h \triangleright (k \triangleright a) = hk \triangleright a$.
\end{itemize}
\end{dfn}
\nd In this case, $A$ is called an \emph{$H$-module algebra}.
\vu

Note that in this definition we do not need to require $A$ to be unital. However, in the literature we find the
definition of an action of a weak bialgebra H on an algebra $A$ with unit $1_A$, where the conditions (i)-(iii) have to be satisfied,
as well as the fourth condition $$h \triangleright 1_A = \varepsilon_{{}_{L}}(h) \triangleright 1_A.$$
Nevertheless, in the case that $H$ is a weak Hopf algebra, this fourth condition is
implied by the three previous ones, as we will see in the following lemma.

\vu

\begin{lema}\label{condexcpma}
Let $A$ be an $H$-module algebra. Then, $$h \triangleright 1_A = \varepsilon_{{}_{L}}(h) \triangleright 1_A,$$ for all $h \in H$.
\end{lema}

\begin{proof} In fact,
$$
	\begin{array}{rcl}
		\varepsilon_{{}_{L}}(h)\triangleright 1_{A}&=&h_1S(h_2)\triangleright 1_{A}\\
		&\overset{\text{(iii)}}{=}&h_1 \triangleright (1_{A}(S(h_2)\triangleright 1_{A}))\\
		&\overset{\text{(ii)}}{=}&(h_1\triangleright 1_{A})(h_2S(h_3)\triangleright 1_{A})\\
		&\overset{\eqref{p4-1}}{=} &(1_1h\triangleright 1_{A})(1_2\triangleright 1_{A})\\
	    &\overset{\text{(iii)}}{=}&(1_1\triangleright(h\triangleright 1_A))(1_2\triangleright 1_{A})\\
		&\overset{\text{(ii)}}{=}&1_H\triangleright(h\triangleright 1_{A}) 1_{A}\\
		&\overset{\text{(i)}}{=}&h\triangleright 1_{A}.
	\end{array}
$$
\end{proof}	
In the setting of partial actions we have the following.

\begin{dfn}\label{acaoparcial}
Let $A$ be an algebra. A \emph{partial action}
of $H$ on $A$ is a $\Bbbk$-linear map
$\cdot\colon H \otimes A \to A$ such that the following properties hold for all $a,b \in A$ and $h,k \in H$:

\vu

\begin{itemize}

\item[(i)] $1_H \cdot a = a$,

\vu

\item[(ii)] $h \cdot ab = (h_1 \cdot a)(h_2 \cdot b)$,

\vu

\item[(iii)] $h \cdot (k \cdot a) = (h_1 \cdot 1_A)(h_2k \cdot a)$.

\end{itemize}

In this case $A$ is called a \emph{partial $H$-module algebra}.

\vu

Moreover, we say that $\cdot$ is \emph{symmetric} (or, $A$ is a \emph{symmetric partial $H$-module algebra}) if the additional condition also holds:

\begin{itemize}

\item[(iv)] $h \cdot (k \cdot a) = (h_1 k \cdot a)(h_2 \cdot 1_A)$.
\end{itemize}
\end{dfn}

\begin{obs}\label{nonunital}
	Observe that, assuming the condition (i), the conditions (ii) and (iii) in Definition \ref{acaoparcial} are
	equivalent to $$h \cdot (a(k \cdot b)) = (h_1 \cdot a)(h_2k \cdot b).$$
	
	In a similar way, conditions (ii) and (iv) are equivalent to $$h \cdot ((k \cdot a)b) = (h_1k \cdot a)(h_2 \cdot b).$$
\end{obs}

It is immediate to check that any (global) action is a particular example of a partial one.
The following proposition tells us under what condition a partial action is global.

\begin{lema}
\label{condglobal}
	Let $A$ be a partial $H$-module algebra. Then, $A$ is an $H$-module algebra if and only if
$h\cdot 1_{A}=\varepsilon_{{}_{L}}(h)\cdot 1_{A}$, for all $h\in H$.
\end{lema}

\begin{proof} Suppose that $h\cdot 1_{A}=\varepsilon_{{}_{L}}(h)\cdot 1_{A}$, for all $h\in H$.
Then,
$$
\begin{array}{rcl}
	h\cdot(g\cdot a) &=& (h_1\cdot 1_{A})(h_2g\cdot a)\\
	&=& (\varepsilon_{{}_{L}}(h_1)\cdot 1_{A})(h_2g\cdot a)\\
	&=& (h_1S(h_2)\cdot 1_{A})(h_3g\cdot a)\\
	&\overset{\eqref{p12-2}}{=}& (S(1_1)\cdot 1_{A})(1_2hg\cdot a)\\
\end{array}
$$
$$
\begin{array}{rcl}
	&=& (\varepsilon_{{}_{L}}(S(1_1))\cdot 1_{A})(1_2hg\cdot a)\\
	&\overset{\eqref{p9-1}}{=}& (\varepsilon_{{}_{L}}(\varepsilon_{{}_{R}}(1_1))\cdot 1_{A})(1_2hg\cdot a)\\
	&=& (\varepsilon_{{}_{R}}(1_1)\cdot 1_{A})(1_2hg\cdot a)\\
	&\overset{\eqref{p2}}{=}& (1_1\cdot 1_{A})(1_2hg\cdot a)\\
	&=& 1_H\cdot(hg\cdot a)\\
	&=& hg\cdot a,
\end{array}
$$ for all $g,h\in H$.
\end{proof}

In the next lemmas we will see some technical properties of partial actions, which will be
very useful tools in the sequel.

\vu

\begin{lema}\label{noHRcola}
	Let $A$ be a partial $H$-module algebra and $x\in H$.  If $w\in H_R$ (or, $w\in H_L$ and the partial action is symmetric),
 then $$w\cdot (h\cdot a)=wh\cdot a,$$ for every $h\in H$ and $a\in A$.
\end{lema}

\begin{proof} Suppose first that $w\in H_R$. Thus,
	\[
	\begin{array}{ccl}
	w\cdot (h\cdot a) &=&(w_1\cdot 1_{A})(w_2h\cdot a)\\
	&\overset{\eqref{p3-2}}{=}&(1_1\cdot 1_{A})(1_2wh\cdot a)\\
	&= &1_H\cdot(wh\cdot a)\\
	&= &wh\cdot a.
	\end{array}
	\]
Now, assuming that the partial action is symmetric and $w\in H_L$, we have
\[
	\begin{array}{ccl}
	w\cdot (h\cdot a) &=&(w_1h\cdot a)(w_2 \cdot  1_{A})\\
	&\overset{\eqref{p3-1}}{=}&(1_1wh\cdot 1_{A})(1_2\cdot a)\\
	&=&1_H\cdot(wh\cdot a)\\
	&=& wh\cdot a.
	\end{array}
	\]
\end{proof}

\begin{lema}
\label{lemaadicional}
Let $A$ be a partial $H$-module algebra, $h,k \in H$ and $a,b \in A$. Then, $$(h \cdot a)(k \cdot b) = (1_1h \cdot a)(1_2 k \cdot b).$$
\end{lema}

\begin{proof} In fact,
$$
\begin{array}{rcl}
	(h \cdot a)(k \cdot b) & = & 1_H \cdot [(h \cdot a)(k \cdot b)] \\
	&\overset{}{=}& (1_1 \cdot h \cdot a)(1_2 \cdot k \cdot b) \\
	&\overset{}{=}& (1_1 \cdot h \cdot a)(1_2 \cdot 1_A)(1_3 \cdot k \cdot b) \\
	&\overset{}{=}& (1_1 \cdot h \cdot a)(1_2 k \cdot b) \\
	&=&  (1_1 h \cdot a)(1_2 k \cdot b)
\end{array}
$$
where the last equality follows by Lemmas \ref{noHRcola} and \ref{p2}.
\end{proof}

\vu

\begin{lema}\label{saiprafora}
	Let $A$ be a partial $H$-module algebra, $a, b\in A$ and $z\in H$.

\begin{itemize}
	\item[(i)] If $z\in H_L$, then $(z\cdot a)b=z\cdot ab$.

\vu

	\item[(ii)] If $z\in H_R$, then $a(z\cdot b)=z\cdot ab$.
\end{itemize}
In particular, $(H_L \cdot A)$ is a right ideal of $A$ and $(H_R \cdot A)$ is a left ideal of $A$.
\end{lema}

\begin{proof}(i) If $z\in H_L$,
	\[
	\begin{array}{rcl}
		(z\cdot a)b&=&(z\cdot a)(1_H\cdot b)\\
		&\Ref{lemaadicional}& (1_1z\cdot a)(1_2\cdot b)\\
		&\Ref{p3-1}&(z_1\cdot a)(z_2\cdot b)\\
		&=&z\cdot ab
	\end{array}
	\]

(ii) If $z\in H_R$,
	\[
	\begin{array}{rcl}
		a(z\cdot b)&\overset{}{=}&(1_H\cdot a)(z\cdot b) \\
		&\Ref{lemaadicional}&(1_1\cdot a)(1_2z\cdot b) \\
		&\overset{\eqref{p3-2}}{=}& (z_1\cdot a)(z_2\cdot b)\\
		&=&z\cdot ab
	\end{array}
    \]

\vu

The last assertion is obvious.
\end{proof}
From the above lemma, we have the following immediate consequences.

\vu

\begin{coro} Let $A$ be a partial $H$-module algebra, $h,z\in H$ and $a\in A$.

\begin{itemize}

	\item[(i)] If $z\in H_L$, we have that $(z\cdot 1_{A})(h\cdot a)=z\cdot (h\cdot a)$. If, in addition,
the action is symmetric, then $z\cdot (h\cdot a)=zh\cdot a$.

\vu

	\item[(ii)] If $z \in H_R$, then $(h \cdot a)(z \cdot  1_{A}) = zh \cdot a$ \qed
\end{itemize}
\end{coro}

\vu

\begin{lema}
\label{2.25}
Let $A$ be a partial $H$-module algebra. The following assertions hold for all $h,k \in H$ and $a,b \in A$:
\begin{itemize}	
	\item[(i)] $(h\cdot a)(k\cdot b) = h_1\cdot (a(S(h_2)k\cdot b))$.

\vu

	\item[(ii)] If the action is symmetric and the antipode $S$ is invertible then, $$(h\cdot a)(k\cdot b) = k_2\cdot (( S^{-1}(k_1)h\cdot a)b).$$
\end{itemize}
	
\end{lema}

\begin{proof}Let $h,k \in H$ and $a,b \in A$. Then,

(i)
\[
	\begin{array}{rcl}
		h_1\cdot a(S(h_2) k \cdot b) &=& (h_1\cdot a)(h_2\cdot (S(h_3) k \cdot b))\\
		&=& (h_1\cdot a)(h_2\cdot  1_{A})(h_3S(h_4) k \cdot b)\\
		&=& (h_1\cdot a)(h_2S(h_3) k \cdot b)\\
		&\overset{\eqref{p4-1}}{=}& (1_1 h\cdot a)(1_2 k \cdot b)\\
		&\Ref{lemaadicional}&(h\cdot a)(k \cdot b).
	\end{array}
\]

\vu

(ii) Since $S$ is invertible, we obtain from \eqref{p12-2} that
\begin{eqnarray}
\label{u1tu2h=h2sih1th3}
k_2 S^{-1}(k_1) \otimes k_3 = 1_1 \otimes 1_2 k,
\end{eqnarray}  for all $k \in H$.

\vu

Thus,
\[
\begin{array}{rcl}
k_2 \cdot (S^{-1}(k_1)h \cdot a)b &=& (k_2 \cdot (S^{-1}(k_1)h \cdot a))(k_3 \cdot b) \\
 &\overset{\textrm{sym}}{=}& (k_2 S^{-1}(k_1)h \cdot a)(k_3 \cdot  1_{A})(k_4 \cdot b) \\
 &=& (k_2 S^{-1}(k_1)h \cdot a)(k_3 \cdot b) \\
 &\Ref{u1tu2h=h2sih1th3}& (1_1 h\cdot a)(1_2 k \cdot b)  \\
 &\Ref{lemaadicional}& (h\cdot a)(k \cdot b)
\end{array}
\]
\end{proof}

%
%

\section{Partial Groupoid Actions}\label{groupoid}

Partial groupoid action was introduced in the literature by D. Bagio and A. Paques in \cite{BP}.
Our main purpose in this section is to prove that, given a groupoid $G$ such that the set $G_0$ of all its identities is finite,
there is a one to one correspondence between the symmetric partial actions of the
groupoid algebra $\Bbbk G$ on a $G_0$-graded algebra $A$ and the
globalizable partial actions of the groupoid $G$ on $A$ (see Theorem \ref{pga}).

\vu

\begin{dfn}
A \emph{groupoid} is a non-empty set $G$ equipped with a partially defined binary operation, for which the
usual axioms of a group hold whenever they make sense, that is:

(i) For every $g, h, l \in G$, $g(hl)$ exists if and only if $(gh)l$ exists and in this
case they are equal.

\vu

(ii) For every $g, h, l\in G$, $g(hl)$ exists if and only if $gh$ and $hl$ exist.

\vu

(iii) For each $g \in G$ there exist (unique) elements $d(g), r(g) \in G$ such
that $gd(g)$ and $r(g)g$ exist and $gd(g) = g = r(g)g$.

\vu

(iv) For each $g \in G$ there exists an element $g^{-1} \in G$ such that $d(g) = g^{-1}g$ and $r(g) = g g^{-1}$.
\end{dfn}

The uniqueness of the element $g^{-1}$ is an immediate consequence of the above definition, and $(g^{-1})^{-1} = g$, for all $g\in G$.
The element $gh$ there exists if and only if $d(g) = r(h)$ if and only if there exists $h^{-1} g^{-1}$
and, in this case, $(gh)^{-1} = h^{-1} g^{-1}$, $r(gh) = r(g)$ and $d(gh) = d(h)$.

An element $e\in G$ is said to be an \emph{identity} of $G$ if there exists $g\in G$ such that $e = d(g)$ (and so $e = r(g^{-1})$).
Let $G_0$ denote the set of all identities of $G$. Note that $e = e^{-1}=d(e) = r(e)$, for all $e\in G_0$.
For more about groupoid's properties we refer to \cite{Lawson}.

\vu

\begin{dfn} \label{pag} \cite{BP} A \emph{partial action of a groupoid} $G$ on an algebra $A$ is a pair
$$\alpha = (\{\alpha_g\}_{g\in G}, \{D_g\}_{g\in G})$$
where, for each $e\in G_0$ and $g\in G$, $D_e$ is an ideal of $A$, $D_g$ is an ideal of $D_{r(g)}$, and $\alpha_g\colon D_{g^{-1}}\to D_{g}$
is an algebra isomorphism such that:

\begin{itemize}

\item[(i)] $\alpha_e$ is the identity map $ I_{D_e}$ of $D_e$,

\vu

\item[(ii)] $\alpha_{h}^{-1}(D_{g^{-1}}\cap D_{h})\subseteq D_{(gh)^{-1}}$,

\vu

\item[(iii)] $\alpha_g(\alpha_h(x)) = \alpha_{gh}(x)$, for all $x\in \alpha_{h^{-1}}(D_{g^{-1}}\cap D_h)$,
\end{itemize}
for all $e\in G_0$ and $g,h\in G$ such that $d(g)=r(h)$.
\end{dfn}


For the proof of the following lemma see \cite{BP}*{Lemma 1.1}.

\vu

\begin{lema}  Let $\alpha = (\{\alpha_g\}_{g\in G}, \{D_g\}_{g\in G})$ be a partial action of a groupoid $G$
on an algebra $A$. Then,

\begin{itemize}
	
\item[(i)]	 $\alpha_g{}^{-1} = \alpha_{g^{-1}}$, for all $g\in G$,
	
\vu

\item[(ii)] $\alpha_g(D_{g^{-1}}\cap D_h) = D_g\cap D_{gh}$, if $d(g)=r(h)$.\qed

\end{itemize}
\end{lema}

\vu

Given a groupoid $G$, the groupoid algebra $\Bbbk G$ is  a $\Bbbk$-vector space with basis $\{\delta_g\ |\ g\in G\}$,
and multiplication
given by the rule
\[
\delta_g\delta_h=
\begin{cases}
\delta_{gh}, &\textrm{if}\ \ d(g)=r(h)\\
0,  &\text{otherwise}
\end{cases}
\]
for all $g,h\in G$. It is easy to see that $\Bbbk G$ is an algebra, and has an identity element, given by $1_{\Bbbk G}=\sum_{e\in
G_0}\delta_e$, if and only if $G_0$ is finite \cite{Lun}. Moreover, $\Bbbk G$ has a coalgebra
structure given by $$\Delta(\delta_g)=\delta_g\otimes \delta_g\quad\text{and}\quad \varepsilon(\delta_g)=1_{\Bbbk},$$
for all $g\in G$. It is well known that $\Bbbk G$, with the algebra and coalgebra
structures above described, and antipode $S$ given by $S(g)=g^{-1}$, for all $g\in G$, is  a weak Hopf algebra.

\vd

From now on we will assume that $G_0$ is finite and $\alpha = (\{\alpha_g\}_{g\in G}, \{D_g\}_{g\in G})$  is a partial action of $G$
on $A = \bigoplus\limits_{e\in G_{0}} D_e$. We also assume that each ideal $D_g$ has a unit, denoted by $1_g$. Notice that, in this case,
each $1_g$ is a central element of $A$ (in particular, $D_g$ is also an ideal of $A$), and the conditions (ii) and (iii) of Definition \ref{pag} imply the following:
\begin{eqnarray}\label{mag}
\alpha_g(\alpha_h(x1_{h^{-1}})1_{g^{-1}})=\alpha_{gh}(x1_{(gh)^{-1}})1_g,
\end{eqnarray}
for all $x\in D_{h^{-1}}\bigcap D_{(gh)^{-1}}$, whenever $d(g)=r(h)$.

\vu

\begin{lema}\label{lema34}
With the notations and assumptions given above, the map
	$$
		\begin{array}{rl}
		 	\cdot\colon \Bbbk G\otimes A &\to A\\
			\delta_{g} \otimes a & \mapsto \alpha_{g}(a 1_{g^{-1}})
		\end{array}
	$$
is a symmetric partial action of $\Bbbk G$ on $A$.	
\end{lema}

\begin{proof}

Indeed, $\cdot$ is a well defined linear map. Furthermore,
	\begin{itemize}
	\item[(i)] for all $a\in A$,
	$$
		\begin{array}{rcl}
			1_{\Bbbk G}\cdot a &=& \sum\limits_{e\in G_0} \delta_e\cdot a\\
			&=& \sum\limits_{e\in G_0} \alpha_e(a 1_{e^{-1}})\\
			&=& \sum\limits_{e\in G_0} a 1_{e}\\
			&=& a  1_{A} = a
		\end{array}
	$$
	\item[(ii)] for all $g \in G$ and $a,b \in A$,
	$$
		\begin{array}{rcl}
			\delta_g \cdot ab &=& \alpha_g (ab1_{g^{-1}})\\
			&=& \alpha_g (a1_{g^{-1}})\alpha_g(b1_{g^{-1}})\\
			&=& (\delta_g \cdot a)(\delta_g\cdot b)
		\end{array}
	$$
	\item[(iii)] for all $g, h\in G$ and $a\in A$, if $d(g)\neq r(h)$ then $D_{g^{-1}}\bigcap D_h=0=\delta_g\delta_h$, and
$$\delta_g\cdot(\delta_h\cdot a) = \alpha_g(\alpha_h(a1_{h^{-1}})1_{g^{-1}})= 0 = (\delta_g\cdot 1_{A}) (\delta_g \delta_h\cdot a).$$
\nd Otherwise, if $d(g)=r(h)$ then
$$
	\begin{array}{rcl}
			\delta_g\cdot(\delta_h\cdot a) &=& \alpha_g(\alpha_h(a 1_{h^{-1}}) 1_{g^{-1}})\\
			&\overset{\eqref{mag}}{=}& \alpha_{gh} (a 1_{(gh)^{-1}})1_g\\
			&=& (\delta_g\cdot  1_{A}) (\delta_{gh} \cdot a )\\
      &=& (\delta_g\cdot 1_A)(\delta_g\delta_h\cdot a).
	\end{array}
$$

\end{itemize}

The symmetry of $\cdot$ is obvious for, as noticed above, $1_g=\delta_g\cdot 1_A$ is central in $A$, for all $g\in G$.
\end{proof}

The converse of Lemma \ref{lema34} is given in the following theorem, which in particular generalizes  \cite{flores}*{Proposition 2.2}.

\begin{teo}\label{pga}
	Let $A$ be an algebra and $G$ a groupoid such that $G_0$ is finite.
The following statements are equivalent:

\begin{itemize}
	\item[(i)] There exists a partial action $\alpha=(\{\alpha_g\}_{g \in G}$, $\{D_g\}_{g \in G})$ of $G$ on $A$
such that the ideals $D_g$ are unital and $A = \bigoplus\limits_{e \in G_0} D_e$.

\vu

	\item[(ii)] $A$ is a symmetric partial $\Bbbk G$-module algebra.
\end{itemize}
\end{teo}
\begin{proof}
(i)$\Rightarrow$(ii) It follows from Lemma \ref{lema34}.

\vd

(ii)$\Rightarrow$(i) Let $D_g=\delta_g\cdot A$, $1_g=\delta_g\cdot 1_A$, and
$\alpha_g\colon D_{g^{-1}}\to D_g$ given by $\alpha_g(x) =\delta_g\cdot x$, for all
$g\in G$ and $x\in D_{g^{-1}}$. We will proceed by steps.

\vu

To show that $\alpha=(\{\alpha_g\}_{g \in G}, \{D_g\}_{g \in G})$ is a partial
action of $G$ on $A$ we need to check that, for every $g\in G$ and $e\in G_0$,
$D_g$ is an ideal of $D_{r(g)}$, $D_e$ is an ideal of $A$,
and $\alpha_g$ is an  algebra isomorphism, which will be done in the steps 1, 2, and 3.
We also show in the step 1 that the ideals $D_g$, $g\in G$, are all unital. In the step 4,
we show that the conditions (i)-(iii) of Definition \ref{pag} hold.
Finally, in the step 5 we show that $A = \bigoplus\limits_{e \in G_0} D_e$.

\vd

{\bf Step 1}: First of all, $1_g$ is a central idempotent of $A$ and $D_g=1_gA$, which implies that $D_g$ is a unital ideal of $A$, for all $g\in G$.

\vu

Indeed, $(1_g)^2=(\delta_g\cdot 1_A)(\delta_g\cdot 1_A)=\delta_g\cdot 1_A=1_g,$
and

$$
\begin{array}{rcl}
1_ga &= &(\delta_g \cdot 1_A)a\\
&= & 1_{\Bbbk G} \cdot (\delta_g \cdot 1_A)a \\
&= & \sum\limits_{e \in G_0} (\delta_e \cdot \delta_g \cdot 1_A)(\delta_e \cdot a) \\
&= & \sum\limits_{e \in G_0} (\delta_e \delta_g \cdot 1_A)(\delta_e \cdot 1_A)(\delta_e \cdot a) \\
&= & \sum\limits_{e \in G_0} (\delta_e \delta_g \cdot 1_A)(\delta_e \cdot a) \\
&= & (\delta_{r(g)}\delta_g \cdot 1_A)(\delta_{r(g)} \cdot a) \\
&= & (\delta_{r(g)g} \cdot 1_A)(\delta_{r(g)} \cdot a) \\
&= & (\delta_{g} \cdot 1_A)(\delta_g \delta_{g^{-1}} \cdot a) \\
&= & \delta_g \cdot \delta_{g^{-1}} \cdot a \\
&= & (\delta_g \delta_{g^{-1}} \cdot a)(\delta_{g} \cdot 1_A) \\
&= & (\delta_{r(g)} \cdot a)(\delta_{r(g) g} \cdot 1_A ) \\
&= & (\delta_{r(g)} \cdot a)(\delta_{r(g)} \delta_g \cdot 1_A ) \\
&= & \sum\limits_{e \in G_0} (\delta_e \cdot a)(\delta_e\ \delta_g \cdot 1_A) \\
&= & \sum\limits_{e \in G_0} \delta_e \cdot (a (\delta_g \cdot 1_A)) \\
&= & 1_{\Bbbk G} \cdot (a(\delta_g \cdot 1_A)) \\
&= & a(\delta_g \cdot 1_A)\\
&= & a 1_g
\end{array}
$$
 Note that the above sequence of equalities gives an important and useful relation for the partial action of $G$ on $A$, that is,
	\begin{eqnarray}
	(\delta_g \cdot 1_A) a = \delta_g \cdot \delta_{g^{-1}} \cdot a = a(\delta_g \cdot 1_A)
	\label{eq:1ga=dgdgia}
	\end{eqnarray}
for all $g\in G$ and $a\in A$, which implies
$$D_g=\delta_g\cdot A=(\delta_g\cdot1_A)(\delta_g\cdot A)\subseteq 1_gA=(\delta_g \cdot 1_A)A \overset{\eqref{eq:1ga=dgdgia}}{=}
\delta_g \cdot \delta_{g^{-1}} \cdot A \subseteq\delta_g \cdot A=D_g,$$ hence $D_g=1_gA$.
The last assertion is immediate.

\vt

{\bf Step 2}: $D_g=D_{r(g)}1_g$, in particular $D_g$  is an ideal of $D_{r(g)}$, for all $g\in G$.

\vu

It follows from \eqref{eq:1ga=dgdgia} and the symmetry of $\cdot$ that
\[
\begin{array}{rcl}
D_g &=& (\delta_g\cdot 1_A)A\\
&\overset{\eqref{eq:1ga=dgdgia}}{=} & \delta_g \cdot \delta_{g^{-1}} \cdot A\\
&=& (\delta_g \delta_{g^{-1}} \cdot A)(\delta_{g} \cdot 1_A)\\
&=& (\delta_{r(g)} \cdot A)(\delta_g \cdot 1_A )\\
&=&  D_{r(g)}1_g
\end{array}
\]

\vt

{\bf Step 3}: $\alpha_g$ is an isomorphism of algebras, for all $g\in G$:

\vu
	
It is clear from the above that $\alpha_g$ is a well defined linear map. Thus, it is enough to show that $\alpha_g$ is
multiplicative and $\alpha_g^{-1}=\alpha_{g^{-1}}$, for all $g\in G$.

\vu
	
For all $a,b \in A$, we have
$$
\begin{array}{rcl}
	\alpha_g(a b 1_{g^{-1}}) &= & \delta_g \cdot ab 1_{g^{-1}} \\
	&= & (\delta_g \cdot a 1_{g^{-1}}) (\delta_g \cdot b 1_{g^{-1}}) \\
	&= & \alpha_g(a 1_{g^{-1}}) \alpha_g(b 1_{g^{-1}})
\end{array}
$$

In order to show that $\alpha_g$ is an isomorphism, we need first to show that $\alpha_e$ is the identity in $D_e$ for all $e \in G_0$.

\vu

Notice that for any $h\in G$ and $a,b\in A$,
\begin{eqnarray}
	(\delta_h \cdot a)b = (\delta_h \cdot a)(\delta_h \cdot 1_A)b
	\Ref{eq:1ga=dgdgia} (\delta_h \cdot a)(\delta_h \cdot \delta_{h^{-1}} \cdot b)
	= \delta_h \cdot (a(\delta_{h^{-1}} \cdot b))
\label{conseq34}
\end{eqnarray}

\vu

Therefore, for $e\in G_0$ and $a\in A$ we have
$$
\begin{array}{rclr}
	a 1_e &=& (1_{\Bbbk G} \cdot a)1_e & \\
	&=& \sum\limits_{e' \in G_0} (\delta_{e'} \cdot a)(\delta_e \cdot 1_A) & \\
	&\overset{(\ref{conseq34})}{=}& \sum\limits_{e' \in G_0} \delta_{e'} \cdot (a(\delta_{e'} \cdot \delta_e \cdot 1_A)) &  (\text{taking} \ b=\delta_e \cdot 1_A )\\
	&=& \sum\limits_{e' \in G_0} \delta_{e'} \cdot (a(\delta_{e'} \delta_e \cdot 1_A)(\delta_{e'} \cdot 1_A)) &\\
	&=& \delta_e \cdot (a(\delta_e \cdot 1_A)) &\\
	&=& \delta_e \cdot a 1_e &\\
    &=& \alpha_e(a1_e)
\end{array}
$$

Finally, for all $a\in D_g$ we have
$$
\begin{array}{rcl}
	\alpha_g (\alpha_{g^{-1}} (a )) &=& \delta_g \cdot (\delta_{g^{-1}} \cdot a )\\
	&=& (\delta_g \delta_{g^{-1}} \cdot a)(\delta_g \cdot 1_A) \\
	&=& (\delta_{r(g)} \cdot a) 1_g \\
    &=& (\delta_{r(g)} \cdot a1_{r(g)}) 1_g\\
	&=& a1_{r(g)}1_g \\
	&=& a
\end{array}
$$

For the equality $\alpha_{g^{-1}}\alpha_g(a)=a$, for all $a\in D_{g^{-1}}$, one proceeds in a similar way.

\vt

{\bf Step 4}: $\alpha = (\{D_g\}_{g \in G},\{\alpha_g\}_{g \in G})$ is a partial groupoid action of $G$ on $A$.

\vu

We only need to check that $\alpha$ satisfies the conditions (i)-(iii) of Definition \ref{pag}.	
	
\vd
	
(i) It follows from Step 3.

\vd

(ii) For all $h,g\in G$ such that $d(g)=r(h)$ and $a\in A$, we have
$$
	\begin{array}{rcl}
		\alpha_{h^{-1}} (a 1_{g^{-1}} 1_h) &=& \alpha_{h^{-1}} (a 1_h)\alpha_{h^{-1}}(1_{g^{-1}} 1_h)\\
		&\overset{\eqref{mag}}{=}& \alpha_{h^{-1}} (a 1_h)  1_{h^{-1}} 1_{(gh)^{-1}}\in D_{(gh)^{-1}},\\
	\end{array}
$$
thus $\alpha_{h^{-1}}(D_{g^{-1}}\cap D_h) \subseteq D_{(gh)^{-1}}$.

\vd

(iii)  It follows from (ii) that if $x\in \alpha_{h^{-1}}(D_{g^{-1}}\cap D_h)$ then $\alpha_h(x)\in D_{g^{-1}}$ and $x = a 1_{h^{-1}} 1_{(gh)^{-1}}$,
for some $a\in A$.
Hence, the elements $\alpha_g\alpha_h(x)$ and $\alpha_{gh}(x)$ there exist and lie in $D_g\bigcap D_{gh}$, and
$$
	\begin{array}{rcl}
		\alpha_g(\alpha_h (x)) &=& \alpha_g(\delta_h \cdot a 1_{h^{-1}} 1_{(gh)^{-1}})\\
		 &=& \delta_g \cdot (\delta_h \cdot a 1_{h^{-1}} 1_{(gh)^{-1}})\\
		 &=& (\delta_g \delta_h \cdot a 1_{h^{-1}} 1_{(gh)^{-1}})(\delta_g\cdot 1_{A})\\
		 &=& (\delta_{gh} \cdot a 1_{h^{-1}} 1_{(gh)^{-1}})(\delta_g\cdot 1_{A})\\
		 &=& \alpha_{gh}(x) 1_g\\
		 &=& \alpha_{gh}(x)
	\end{array}
$$

\vt

{\bf Step 5}:  $A=\bigoplus\limits_{e\in G_0}D_e$.

\vu

Indeed, notice that $$A=1_{\Bbbk G} \cdot A=
\sum\limits_{e \in G_0} \delta_e \cdot A=\sum\limits_{e \in G_0} D_e,$$

and, since $$1_e1_f=(\delta_e\cdot 1_A)(\delta_f\cdot 1_A)\overset{\eqref{eq:1ga=dgdgia}}{=}\delta_e\cdot (\delta_e\cdot\delta_f\cdot1_A)=
\delta_e\cdot((\delta_e\cdot1_A)(\delta_e\delta_f\cdot 1_A))=0,$$ for all $e\neq f$ in $G_0$, it easily follows that
$D_e \cap (\sum\limits_{\stackrel{f \in G_0}{f \not = e}} D_{f})=0$.
\end{proof}

%
%

\section{partial Actions on the Ground Field}\label{field}

Partial actions of a weak Hopf algebra $H$ on the ground field $\Bbbk$ provide a large amount of examples
of partial actions. In this section we give the necessary and sufficient conditions for an action of $H$ on
$\Bbbk$ to be partial and, as an application, we describe all the partial actions of a groupoid algebra on
 $\Bbbk$.

\vu

It is clear that any action of $H$ on $\Bbbk$ is, in particular, a  $\Bbbk$-linear map from $H$ on  $\Bbbk$.
The question is: under what conditions a  $\Bbbk$-linear map from $H$ on  $\Bbbk$ defines a partial action of $H$ on  $\Bbbk$?
The answer is given in the following lemma.

\begin{lema}
\label{pcorpo}
Let $\lambda\colon H\to\Bbbk$ be a $\Bbbk$-linear map. Then,  $\lambda$ defines a partial action of $H$ on $\Bbbk$,
via $$h\cdot 1_\Bbbk=\lambda(h),\,\,\ \text{for all}\,\,\ h\in H,$$ if and only if
$$\lambda(1_H) = 1_{\Bbbk}\quad\text{and}\quad \lambda(h)\lambda(g) = \lambda(h_1) \lambda(h_2 g),\,\,\ \text{for all}\,\,\ g,h\in H.$$
\end{lema}

\begin{proof}
Assume that $\cdot$ is a partial action. Then, $$\lambda(1_H) = 1_H \cdot 1_{\Bbbk} = 1_{\Bbbk}$$ and
$$\lambda(h) \lambda(g) = h \cdot (g \cdot 1_{\Bbbk}) =
(h_1 \cdot 1_{\Bbbk} )(h_2 g \cdot 1_{\Bbbk}) = \lambda(h_1) \lambda(h_2g).$$

\vu

Conversely, note that $1_H \cdot 1_{\Bbbk} = \lambda(1_H) = 1_{\Bbbk}$ and $h \cdot  (g \cdot 1_{\Bbbk}) = \lambda(h)\lambda(g) =
\lambda(h_1) \lambda(h_2g) = (h_1 \cdot 1_{\Bbbk})(h_2g \cdot 1_{\Bbbk})$. Taking $g = 1_H$ in this last equality we have the
third required condition.
\end{proof}

It is well known that, in the setting of Hopf algebra actions, the only global action on $\Bbbk$ is given by the counit $\varepsilon$.
This is not true for actions of weak Hopf algebras. In the following proposition we will give necessary and sufficient conditions
to obtain a global action of a weak Hopf algebra $H$ on $\Bbbk$.

\vu

\begin{prop}
Let  $\lambda\colon H\to \Bbbk$ be a $\Bbbk$-linear map and $\triangleright\colon H\otimes \Bbbk\to \Bbbk$ the $\Bbbk$-linear map
given by $h\triangleright 1_{\Bbbk}=\lambda(h)$. Then, $\triangleright$ is a global action of $H$ on $\Bbbk$ if and only
if $\lambda$ is a convolutional idempotent in $Alg_{\Bbbk}(H,\Bbbk)=\{f\in Hom_{\Bbbk}(H,\Bbbk) \mid f\ \text{is multiplicative}\}$.
\end{prop}

\begin{proof}
Assume that $\triangleright$ is global. Then,  $\lambda(h)\lambda(g) = h \triangleright g \triangleright 1_{\Bbbk} =
h g \triangleright 1_{\Bbbk} = \lambda(hg)$ and $\lambda(1_H) = 1_H \triangleright 1_{\Bbbk} = 1_{\Bbbk}$, for all $g,h\in H$. Thus,
$\lambda \in Alg_{\Bbbk}(H,\Bbbk)$. Moreover, $\lambda \ast \lambda (h) = \lambda(h_1)\lambda(h_2) =
(h_1 \triangleright 1_{\Bbbk})(h_2 \triangleright 1_{\Bbbk}) = h \triangleright 1_{\Bbbk} = \lambda(h)$, for all $h\in H$, that is,
$\lambda \ast \lambda = \lambda$.

\vu

Conversely, since $\lambda$ is a map of $\Bbbk$-algebras, for all $a,b$ in $\Bbbk$ and $g,h$ in $H$, we have
\begin{enumerate}
\item[(i)] $1_H \triangleright 1_{\Bbbk} = \lambda(1_H) = 1_{\Bbbk}$,

\vu

\item[(ii)] $h \triangleright ab = ab \lambda(h) = ab (\lambda \ast \lambda)(h) = (h_1 \triangleright a)(h_2 \triangleright b)$,

\vu

\item[(iii)] $h \triangleright g \triangleright a = a \lambda(h) \lambda(g) = a \lambda(hg) = hg \triangleright a $.
\end{enumerate}
\end{proof}
The example below illustrates this previous result.

\vu

\begin{ex}
Let $G$ be a groupoid given by a disjoint union of finite groups $G_1,...,G_n$.
Choose $G_j$ one of these subgroups and define $\lambda\colon {\Bbbk}G \to {\Bbbk}$ by $\lambda(g) = 1$
if $g \in G_j$ and $\lambda(g) = 0$ otherwise. It is straightforward to check
that $\lambda$ is a convolutive idempotent in $Alg_{\Bbbk}(\Bbbk G,\Bbbk)$.\qed
\end{ex}

Note that the counit $\varepsilon$ of a weak Hopf algebra is not an algebra homomorphism, so it
does not turn $\Bbbk$ on an $H$-module algebra. The next proposition gives a necessary and sufficient
condition for $\varepsilon$ to define a partial (so, global) action on $\Bbbk$.

\begin{prop}
A weak Hopf algebra $H$ is a Hopf algebra if and only if $\varepsilon$ defines a partial action on $\Bbbk$. In this case, $\varepsilon$
is the unique convolutional idempotent in $Alg_{\Bbbk}(H,\Bbbk)$.
\end{prop}

\begin{proof}
If $H$ is a Hopf Algebra, then $\varepsilon$ is an idempotent element in $Alg_{\Bbbk}(H,\Bbbk)$ and
so it defines an action on $\Bbbk$. Conversely,
if $\varepsilon$ defines a partial action on $\Bbbk$, then $\varepsilon (1_H)=1_{\Bbbk}$ and
$\varepsilon(h)\varepsilon(g) = \varepsilon(h_1)\varepsilon(h_2g) =
\varepsilon(\varepsilon(h_1)h_2g) = \varepsilon(hg)$, which implies that $H$ is a Hopf algebra.

\vu

For the last assertion, it is enough to see that if $H$ is a Hopf algebra then $Alg_{\Bbbk}(H,\Bbbk)$ is a group with the convolution product
and $\varepsilon$ is its unit.
\end{proof}

We end this section presenting a complete description of all partial actions of a groupoid algebra $\Bbbk G$ on $\Bbbk$.

\vu

\begin{prop}
Let $G$ be a groupoid such that $G_0$ is finite, and $\lambda\colon \Bbbk G\to \Bbbk$ a $\Bbbk$-linear map.
Then, $\lambda$ defines a partial action of $\Bbbk G$ on $\Bbbk$, as characterized in Lemma 4.1, if and only if the set
$$V = \{ v \in G \ | \ \delta_v \cdot 1_{\Bbbk} = 1_{\Bbbk}=\delta_{r(v)}\cdot 1_{\Bbbk} \}$$
is a group and $\delta_g\cdot 1_\Bbbk =0$, for all $g\in G\setminus V$.
\end{prop}

\begin{proof}
Assume that  $\lambda$ defines a partial action on $\Bbbk$, given by $\delta_g\cdot 1_\Bbbk=\lambda(g)$,
for all $g\in G$. Then, it follows from the equality
$$(\delta_g\cdot 1_\Bbbk)(\delta_h\cdot 1_\Bbbk)= (\delta_g\cdot1_\Bbbk)(\delta_g\delta_h\cdot1_\Bbbk),$$ that
$$
\begin{array}{rcl}
\delta_g\cdot1_{\Bbbk} &=& (\delta_g\cdot1_{\Bbbk})( 1_{\Bbbk G}\cdot 1_{\Bbbk})\\
&=& \sum\limits_{e\in G_0}(\delta_g\cdot1_{\Bbbk})(\delta_e\cdot1_{\Bbbk})\\
&=& \sum\limits_{e\in G_0}(\delta_g\cdot1_{\Bbbk})(\delta_g\delta_e\cdot1_{\Bbbk})\\
&=& (\delta_g\cdot1_{\Bbbk})(\delta_g\delta_{d(g)}\cdot1_{\Bbbk})\\
&=& (\delta_g\cdot1_{\Bbbk})^2.
\end{array}
$$

\vu

Thus,  $\delta_g\cdot1_\Bbbk$ is an idempotent in $\Bbbk$,
and therefore equal to either $1_\Bbbk$ or $0$, for all $g\in G$. Furthermore, the equality $\lambda(1_{\Bbbk G})=1_\Bbbk$
ensures that $V\neq\emptyset$.

\vu

Now, we show that $V$ is a group.

\vu

(i) For all $g,h\in V$, the product $gh$ exists and lies in $V$. This is an immediate consequence of the following expression
$$1_{\Bbbk}=\lambda(\delta_g) \lambda(\delta_h) =\lambda(\delta_g)\lambda(\delta_g \delta_h)=\lambda(\delta_g\delta_h).$$

\vu

(ii) For all $g\in V$, the element $g^{-1}$ lies in $V$. Indeed, since $g\in V$ we have
$$\lambda(\delta_{g^{-1}})=\lambda(\delta_g) \lambda(\delta_{g^{-1}}) =
\lambda(\delta_g) \lambda(\delta_{g g^{-1}}) = \lambda(\delta_g)\lambda(\delta_{r(g)}) = 1_{\Bbbk}.$$

\vu

Conversely, assume that $V$ is group and let $e_V$ denote its identity element. Also, assume that
$\delta_g\cdot 1_\Bbbk = 0$, for all $g\in G\setminus V$. Under these assumptions we have, in particular, that
$r(g)=d(g)=e_V$, for all $g\in V$,
and $\delta_e\cdot1_\Bbbk=0$, for all $e\in G_0$, $e\neq e_V$. Thus,
$\lambda(1_{\Bbbk G})=\sum_{e\in G_0}\lambda(\delta_e)=\lambda(\delta_{e_V})=1_\Bbbk,$ and
it is straightforward to check that $\lambda(\delta_g)\lambda(\delta_h)=\lambda(\delta_g)\lambda(\delta_g\delta_h)$.
\end{proof}

\begin{ex}
Let $G = G_1 \cup G_2$ be the groupoid given by the disjoint union of two groups $G_1$ and $G_2$.
Any subgroup $V$ of $G_1$ (or $G_2$) defines a partial action of $\Bbbk G$ on $\Bbbk$, given by $\lambda(\delta_g) = \delta_{g,V}$, for all $g$ in $G$, where
$\delta_{g,V} = 1_{\Bbbk}$ if $g \in V$ and $\delta_{g,V} = 0$ otherwise.
\end{ex}

\section{Globalization  of Partial Actions}\label{glob}

In this section we show that any partial action of a weak Hopf algebra can be obtained from a global one. Particularly
in this section, the notation $\cdot$ for partial actions and $\rhd$ for global ones is crucial.

\vu

First of all, given a global action we will see how to construct a partial one from it. The method to do this is given
in the following lemma.

\begin{lema}\label{induzida}
Let  $B$ be an $H$-module algebra via $\triangleright\colon H\otimes B\to B$.
Let $A$ be a right ideal of $B$ which is also  an algebra with unit $1_A$.
Then, the $\Bbbk$-linear map $\cdot\colon H\otimes A\to A$ given by $$h \cdot a =  1_{A} (h \triangleright a)$$ is a partial action of $H$ on $A$.
\end{lema}

\begin{proof} {$ { } $} For every $a,b \in A$ and $g, h \in H$, we have

\begin{itemize}
	\item[(i)] $1_H \cdot a =  1_{A} (1_H \triangleright a) =  1_{A} a =a$

\vu

	\item[(ii)] $h \cdot (ab) =  1_{A} (h \triangleright ab) =  1_{A} (h_1 \triangleright a)(h_2 \triangleright b) =
1_{A} (h_1 \triangleright a)  1_{A} (h_2 \triangleright b) = (h_1 \cdot a)(h_2 \cdot b)$

\vu

	\item[(iii)] $h \cdot (g \cdot a) =  1_{A} (h \triangleright (g \cdot a)) =  1_{A} (h \triangleright  1_{A} (g \triangleright a)) =
1_{A} (h_1 \triangleright  1_{A}) (h_2 g \triangleright a) =  1_{A} (h_1 \triangleright  1_{A})  1_{A} (h_2 g \triangleright a) =
(h_1 \cdot  1_{A}) (h_2 g \cdot a))$
\end{itemize}
\end{proof}

The partial action $\cdot$ of $H$ on $A$, obtained by the method given above, is called
\emph{induced} by the action $\triangleright$.

\begin{dfn}
Let $A$ be a partial $H$-module algebra. We say that a pair $(B,\theta)$ is
a \emph{globalization of $A$} if $B$ is an $H$-module algebra via $\triangleright\colon H\otimes B\to B$, and

\begin{itemize}
\item[(i)] $\theta \colon A \to B$ is a monomorphism of algebras such that $\theta(A)$ is a right ideal of $B$,

\vu

\item[(ii)] the partial action on $A$ is equivalent to the partial action induced by $\triangleright$ on $\theta(A)$, that is,
$\theta(h \cdot a) = h \cdot \theta(a) = \theta( 1_{A})(h \triangleright \theta(a))$,

\vu

\item[(iii)] $B$ is the $H$-module algebra generated by $\theta(A)$, that is, $B = H \triangleright \theta(A)$.
\end{itemize}
\end{dfn}

\vu

Notice that in the above definition as well as in Lemma \ref{induzida} we do not need to require $B$ to be unital. 

\vu
The existence of such a globalization will be ensured by the construction presented in the sequel.

\vu

We start by taking the convolution algebra $\mathcal{F} = Hom (H,A)$, which is an $H$-module algebra with the action
given by $(h \triangleright f)(k) = f(kh)$, for all $f\in \mathcal{F}$ and $h,k\in H$. Let $\varphi\colon A\to \mathcal{F}$ be the
map given by $\varphi(a)\colon h\mapsto h\cdot a$, for all $a\in A$ and $h\in H$. Put $B=H\triangleright\varphi(A)$.

\vu

\begin{prop}
The pair $(B,\varphi)$ is a globalization of $A$.
\end{prop}

\begin{proof}\

\vu

(i)\, $\varphi$ is an algebra monomorphism such that $\varphi(h \cdot a) =
\varphi( 1_{A}) \ast (h \triangleright \varphi(a)),$ for all $h \in H$ and $a \in A$. Indeed,
\begin{itemize}
\item[-] $\varphi$ is clearly $\Bbbk$-linear,

\vu

\item[-] $\varphi$ is injective because $1_H \cdot a = a$,

\vu

\item[-] $\varphi(ab)(h) = h \cdot ab = (h_1 \cdot a)(h_2 \cdot b) =
\varphi(a)(h_1) \varphi(b)(h_2) = [\varphi(a) \ast \varphi(b)] (h),$ for all $a,b\in A$ and $h\in H$,

\vu

\item[-] $(\varphi( 1_{A}) \ast (h \triangleright \varphi(a)))(k) = (k_1 \cdot  1_{A})(k_2 h \cdot a) =
k \cdot h \cdot a = \varphi(h \cdot a)(k)$, for all $k\in H$.
\end{itemize}

\vd

(ii)\, $\varphi(A)$ is a right ideal of $B$. Indeed,
\[
\begin{array}{rcl}
\varphi(b) \ast (h \triangleright \varphi(a)) & =& \varphi(b) \ast \varphi( 1_{A}) \ast (h \triangleright \varphi(a)) \\
 & =& \varphi(b) \ast \varphi(h \cdot a) \\
 & =& \varphi(b(h \cdot a)),
\end{array}
\] for all $a,b\in A$ and $h\in H$.

\vd

(iii) $B$ is an $H$-module algebra.

\vu

In fact, $B$ is clearly a vector subspace of $\mathcal{F}$ which is invariant under the action $\triangleright$, and
$$
\begin{array}{rcl}
(h \triangleright \varphi(a))\ast(k \triangleright \varphi(b)) & \Ref{2.25}& h_1 \triangleright (\varphi(a) \ast (S(h_2)k \triangleright \varphi(b))) \\
	& =& h_1 \triangleright (\varphi(a) \ast \varphi( 1_{A}) \ast (S(h_2)k \triangleright \varphi(b))) \\
	& =& h_1 \triangleright (\varphi(a) \ast \varphi(S(h_2)k \cdot b)) \\
	& =& h_1 \triangleright (\varphi(a (S(h_2)k \cdot b))),
\end{array}
$$ for all $a,b\in A$ and $h,k\in H$.
\end{proof}

\vd

The pair  $(B, \varphi)$, as constructed above, is called the \emph{standard globalization} of $A$.

\begin{prop}\label{idealstd}
With the above notations, a partial action on $A$ is symmetric if and only if $\varphi(A)$ is an ideal of $B$.
\end{prop}

\begin{proof}
Suppose that a partial action of $H$ on $A$ is symmetric. Then,
$$
	\begin{array}{rcl}
		((h \triangleright \varphi(a)) \ast \varphi(b))(k) & = & (h \triangleright \varphi(a))(k_1)  \varphi(b)(k_2) \\
		& = & \varphi(a)(k_1 h)  \varphi(b)(k_2) \\
		& = & (k_1 h \cdot a) (k_2 \cdot b) \\
		& \overset{\textrm{sym}}{=} & k \cdot ((h \cdot a) b) \\
		& = & \varphi((h \cdot a) b)(k) \\
	\end{array}
$$
for all $h,k \in H$ and $a,b \in A$.  As $\varphi(A)$ is a right ideal of $B$, the required follows.

\vu

Conversely, since $\varphi(A)$ is an ideal of $B$ we have that $\varphi(1_A)$ is central in $B$. Then,

$$
	\begin{array}{rcl}
		k \cdot (h \cdot a) &= & \varphi(h \cdot a)(k) \\
		&= & (h \cdot \varphi(a))(k) \\
		&= & (\varphi(1_A) \ast (h \triangleright \varphi(a)))(k) \\
		&= & ((h \triangleright \varphi(a)) \ast \varphi(1_A))(k) \\
		&= & (h \triangleright \varphi(a))(k_1) \varphi(1_A)(k_2) \\
		&= & \varphi(a)(k_1 h) \varphi(1_A)(k_2) \\
		&= & (k_1 h \cdot a) (k_2 \cdot 1_A) \\
	\end{array}
$$
for all $h,k \in H$ and $a \in A$.
\end{proof}

By a homomorphism between two globalizations of a same partial $H$-module algebra we mean a multiplicative linear map
that commutes with the respective actions. If such a homomorphism is bijective we say that such globalizations are \emph{equivalent}.

\begin{prop}\label{epiforstand}
With the above notations, any globalization of $A$ is a homomorphic preimage of the standard one.
\end{prop}

\begin{proof}
Let $(B', \theta)$ be a globalization of the partial $H$-module algebra $A$ and define the following map
$$
\begin{array}{rccc}
	\Phi\colon &B' & \to & B\\
	&\sum\limits_{i=0}^{n} h_i \triangleright \theta(a) & \mapsto & \sum\limits_{i=0}^{n} h_i \triangleright \varphi(a)
\end{array}
$$
In order to prove that $\Phi$ is well defined, it is enough to check that
if $\sum\limits_{i=0}^{n} h_i \triangleright \theta(a) = 0$ then $\sum\limits_{i=0}^{n} h_i \triangleright \varphi(a) = 0$.
Assume that $\sum\limits_{i=0}^{n} h_i \triangleright \theta(a) = 0$. Then, for all $k\in H$ we have
$$
\begin{array}{rcl}
	0 &=& \theta(1_A) (k \triangleright \sum\limits_{i=0}^{n} h_i \triangleright \theta(a)) \\
	&=& \theta(1_A) (\sum\limits_{i=0}^{n} k h_i \triangleright \theta(a)) \\
	&=& \sum\limits_{i=0}^{n} k h_i \cdot \theta(a) \\
	&=& \theta\left(\sum\limits_{i=0}^{n} k h_i \cdot a\right) \\
\end{array}
$$
and, as $\theta$ is injective, we get $\sum\limits_{i=0}^{n} k h_i \cdot a = 0$.

Hence, for any $k\in H$,

$$
\begin{array}{rcl}
	\left(\sum\limits_{i=0}^{n} h_i \triangleright \varphi(a)\right) (k) &=& \sum\limits_{i=0}^{n} \varphi(a) (k h_i)\\
	&=& \sum\limits_{i=0}^{n} k h_i \cdot a\\
	& =& 0.
\end{array}
$$

Clearly, $\Phi$ is surjective and $\Phi(g\triangleright b')=g\triangleright\Phi(b')$, for all $b'\in B'$ and $g\in H$.

\vu

Finally, for all $h,k\in H$ and $a,b\in A$,

$$
\begin{array}{rcl}
	\Phi((h\triangleright \theta(a))(k\triangleright \theta(b))) & = & \Phi(1_H \triangleright ((h\triangleright \theta(a))(k\triangleright \theta(b))))\\
	& = & \Phi((1_1 \triangleright h\triangleright \theta(a))(1_2 \triangleright k\triangleright \theta(b)))\\
	& = & \Phi((1_1 h\triangleright \theta(a))(1_2 k\triangleright \theta(b)))\\
	& = & \Phi((h_1 \triangleright \theta(a))(h_2 S(h_3) k\triangleright \theta(b)))\\
	& = & \Phi(h_1 \triangleright (\theta(a)(S(h_2) k\triangleright \theta(b))))\\
	& = & \Phi(h_1 \triangleright \theta(a(S(h_2) k\cdot b)))\\
	& = & h_1 \triangleright (\varphi(a(S(h_2) k\cdot b)))\\
	& = & h_1 \triangleright (\varphi(a)\ast(S(h_2) k\triangleright \varphi(b)))\\
	& = & (h_1 \triangleright \varphi(a))\ast(h_2 S(h_3) k\triangleright \varphi(b))\\
	& = & (1_1 h \triangleright \varphi(a))\ast(1_2 k\triangleright \varphi(b))\\
	& = & (1_1 \triangleright h \triangleright \varphi(a))\ast(1_2 \triangleright k\triangleright \varphi(b))\\
	& = & 1_H \triangleright ((h \triangleright \varphi(a))\ast(k\triangleright \varphi(b)))\\
	& = & \Phi(h \triangleright \theta(a))\ast\Phi(k\triangleright \theta(b))
\end{array}
$$
\end{proof}


\begin{dfn}
Let $(B,\theta)$ be a globalization of a partial $H$-module algebra $A$. We say that $B$ is \emph{minimal} if for
every $H$-submodule $M$ of $B$ such that $\theta( 1_{A})M = 0$ we have $M = 0$.
\end{dfn}

\begin{prop}
	The standard globalization $(B, \varphi)$ of $A$ is minimal.
\end{prop}

\begin{proof}
	It is enough to prove that the minimal condition holds for any cyclic submodule of $B$. Let $m = \sum\limits_{i=0}^n h_i\triangleright\varphi(a_i)$
be an element in $B$. Suppose that $\varphi( 1_{A})\ast\langle m\rangle = 0$, where $\langle m\rangle$ is the $H$-submodule of
$B$ generated by $m$, that is, $\langle m\rangle = H\triangleright m$.

\vu
	
	Then, for all $k\in H$,
	
	$$
		\begin{array}{rcl}
			0 &=& \varphi( 1_{A})\ast(k\triangleright m)\\
			&=& \varphi( 1_{A})\ast\left(k\triangleright\sum\limits_{i=0}^n h_i\triangleright\varphi(a_i)\right)\\
			&=& \varphi( 1_{A})\ast\left(\sum_i k h_i\triangleright\varphi(a_i)\right)\\
			&=& (\sum_i k h_i\cdot\varphi(a_i))\\
			&=& \varphi(\sum_i k h_i\cdot a_i)
		\end{array}
	$$
	
	which implies $\sum\limits_i k h_i\cdot a_i = 0$, since $\varphi$ is a monomorphism.
	
	Since $m\in B\subseteq Hom(H, A)$ we have	
	$$
		\begin{array}{rcl}
			m(k) &=& (\sum\limits_{i=0}^n h_i\triangleright\varphi(a_i))(k)\\
			&=& (\sum_i \varphi(a_i))(k h_i)\\
			&=& \sum_i k h_i\cdot a_i\\
			&=& 0
		\end{array}
	$$
for all $k\in H$. Therefore, $m=0$.
\end{proof}

\begin{prop}
\label{minisostd}
Any two minimal globalizations of a partial $H$-module algebra $A$ are equivalent.
\end{prop}

\begin{proof}
Let $(B',\theta)$ be a minimal globalization of $A$, $(B,\varphi)$ the standard one, and $\Phi\colon B'\to B$ as defined
in Proposition \ref{epiforstand}. It is enough to prove that $\Phi$ is injective.

\vu

Suppose $\Phi(\sum_i h_i \triangleright \theta(a_i)) = 0$. Thus, $0 = (\sum_i{h_i \triangleright \varphi(a_i)}) (g) = \sum_i{g h_i \cdot a_i}$,
for all $g \in H$, and so $0 = \theta(\sum_i{g h_i \cdot a_i}) = \sum_i{g h_i \cdot \theta(a_i)} = \theta( 1_{A})(\sum_i{g h_i \triangleright \theta(a_i)}) =
\theta( 1_{A})(g \triangleright \sum_i{h_i \triangleright \theta(a_i)})$.

Now, if $M$ denotes the $H$-submodule of $B'$ generated by
$\sum_i{h_i \triangleright \theta(a_i)}$ we have that $\theta( 1_{A})M = 0$, hence $M = 0$. Therefore, $\Phi$ is injective.
\end{proof}

We end this section summarizing all the above main results in the following theorem.

\begin{teo} Let $A$ be a partial $H$-module algebra.
\begin {itemize}
\item[(i)] $A$ has a minimal globalization.

\vu

\item[(ii)] Any two minimal globalization of $A$ are equivalent.

\vu

\item[(iii)] Any globalization of $A$ is a homomorphic preimage of a minimal one.
\end{itemize}\qed
\end{teo}

%
%

\section{Partial Smash Product}\label{smash}

In this section we construct the smash product for a partial $H$-module algebra.

\vu

The smash product already exists
for an $H$-module algebra and even for a partial $H$-module algebra when $H$ is a Hopf algebra.
So, it is natural to ask if it still works for a partial $H$-module algebra when $H$ is a weak Hopf algebra.
The hard task here is to get the good definition of smash product. In fact, the smash product for
Hopf algebra actions (partial or global) is, by construction, a tensor product over the ground field, which makes
easy to show that it is well defined. This does not occur when dealing with weak Hopf algebra actions because
the tensor product, in this case, is not anymore over the ground field but over the algebra $H_L$.

\vu

In order to get our aim we need first a right $H_L$-module structure for a partial $H$-module algebra.

\begin{prop}\label{hhlmoddir}
	Let $A$ be a partial $H$-module algebra. Then, $A$ is a right $H_L$-module via $a\triangleleft z= S_R^{-1}(z)\cdot a = a( S_R^{-1}(z)\cdot 1_{A}),$
for all $a\in A$ and $z\in H_L$.
\end{prop}

\begin{proof}
	As $1_H\in H_L$ and $ S_R^{-1}(1_H)=1_H$, it follows that $a \triangleleft 1_H = 1_H \cdot a = a$
	
	Let $a\in A$, $h,g\in H_L$, so
	
	$$
	\begin{array}{rcl}
		(a\triangleleft h)\triangleleft g&=&  S_R^{-1}(g)\cdot( S_R^{-1}(h)\cdot a)\\
		&=&( S_R^{-1}(g)_1\cdot 1_{A})( S_R^{-1}(g)_2 S_R^{-1}(h)\cdot a)\\
		&\overset{\eqref{p3-2}}{=}& (1_1\cdot 1_{A})(1_2 S_R^{-1}(g) S_R^{-1}(h)\cdot a)\\
		&=&1_H\cdot( S_R^{-1}(hg)\cdot a)\\
		&=&a\triangleleft hg.
	\end{array}
	$$
	
	The equality $ S_R^{-1}(z)\cdot a = a( S_R^{-1}(z)\cdot 1_{A})$ holds by \ref{saiprafora}(i).
\end{proof}

Notice that the  action, by restriction, of $H_R$ on a partial $H$-module algebra usually behaves like a global action.
The Proposition $\ref{noHRcola}$ is a good example of it. In fact, a partial $H$-module algebra
does not become an $H_R$-module algebra simply because $H_R$ is not a coalgebra. However, we can
still consider $H_R$ acting in a similar way as a global action. The next lemma shows one more property for
the action of $H_R$ on a partial $H$-module algebra that works like a global action.

\vu

\begin{lema}
\label{noHRsomeepsilon}
 Let $A$ be a partial $H$-module algebra.
	If $h$ belongs to $H_R$, then $\varepsilon_{{}_{L}}(h)\cdot 1_{A}=h\cdot 1_{A}$.
\end{lema}

\begin{proof} Let $h\in H_R$
		$$
	\begin{array}{rcl}
		\varepsilon_{{}_{L}}(h)\cdot 1_{A}&=&h_1S(h_2)\cdot 1_{A}\\
		&\overset{\eqref{p3-2}}{=}& 1_1S(1_2h)\cdot 1_{A} \\
		&\Ref{noHRcola}&1_1\cdot (S(1_2h)\cdot 1_{A}) \\
		&\Ref{p3-2}&h_1\cdot(S(h_2)\cdot 1_{A})\\
		&=&(h_1\cdot 1_{A})(h_2S(h_3)\cdot 1_{A})\\
		&\Ref{p4-1}&(1_1h\cdot 1_{A})(1_2\cdot 1_{A})\\
		&\Ref{noHRcola}&(1_1\cdot(h\cdot 1_{A}))(1_2\cdot 1_{A})\\
		&=&1_H\cdot(h\cdot 1_{A}) 1_{A}\\
		&=&h\cdot 1_{A}.
	\end{array}
	$$
	\end{proof}

There is also another useful characterization for the right action of $H_L$ on $A$.

\begin{lema}
\label{acaodireita}
	If $z\in H_L$, then $a \triangleleft z=a(z\cdot 1_{A})$.
\end{lema}

\begin{proof}Let $z\in H_L$.
	$$
	\begin{array}{rclr}
		a\triangleleft z&=& a( S_R^{-1}(z)\cdot 1_{A})&\\
		&\Ref{noHRsomeepsilon}& a(\varepsilon_{{}_{L}}( S_R^{-1}(z))\cdot 1_{A})&\\
		&=& a(\varepsilon_{{}_{L}}(\varepsilon_{{}_{R}}( S_R^{-1}(z)))\cdot 1_{A})&\ \mathrm{since} \  S_R^{-1}(z)\in H_R\\
		&\overset{\eqref{p9-1}}{=}& a(\varepsilon_{{}_{L}}(S( S_R^{-1}(z)))\cdot 1_{A})&\\
		&=& a(\varepsilon_{{}_{L}}(z)\cdot 1_{A})&\\
		&=& a(z\cdot 1_{A})&\  \mathrm{since} \ z\in H_L.
	\end{array}
	$$
\end{proof}

Now we are able to define the smash product for a partial $H$-module algebra $A$.

\vu

First, notice that $H$ has a natural structure of a left $H_L$-module via its multiplication.
We start by considering the $\Bbbk$-vector space
given by the tensor product $A\otimes_{H_{_{L}}} H$, and also denoted by $A\# H$,
with the multiplication defined by $$(a\# h)(b\# g)=a(h_1\cdot b)\# h_2g.$$

\begin{teo}
This above multiplication is well defined, associative, and $1_A\#1_H$ is a left unit.
\end{teo}

\begin{proof}
The well definition:

It is enough to show that the map $\tilde{\mu}\colon A\times H\times A\times H\to A\# H$ given by $\tilde{\mu}(a,h,b,g)=a(h_1\cdot b)\# h_2g$
is $(H_L,\Bbbk,H_L)$-balanced. In fact, for all $a,b \in A$, $h,g \in H$, $z \in H_L$ and $r \in \Bbbk$ we have:
	
$$
	\begin{array}{rcl}
		\tilde{\mu}(a,h,b\triangleleft z, g)&=& a(h_1\cdot(b\triangleleft z))\#h_2g\\
		&=& a(h_1\cdot( S_R^{-1}(z)\cdot b))\#h_2g\\
		&=& a(h_1\cdot 1_{A})(h_2 S_R^{-1}(z)\cdot b)\#h_3g\\
		&=& a(h_1\cdot 1_{A})(h_21_1 S_R^{-1}(z)\cdot b)\#h_31_2g\\
		&\overset{\eqref{p13-1}}{=}& a(h_1\cdot 1_{A})(h_21_1\cdot b)\#h_31_2zg\\
		&=& a(h_1\cdot 1_{A})(h_2\cdot b)\#h_3zg\\
		&=& a(h_1\cdot 1_{A} b)\#h_2zg\\
		&=& a(h_1\cdot b)\#h_2zg\\
		&=&\tilde{\mu}(a,h,b,zg).
	\end{array}
$$
	
It is clear that $\tilde{\mu}(a,h r,b, g)=\tilde{\mu}(a,h, rb, g)$, and
	
	$$
	\begin{array}{rcl}
		\tilde{\mu}(a\triangleleft z,h,b, g)&=&(a\triangleleft z)(h_1\cdot b)\#h_2g\\
		&=&a( S_R^{-1}(z)\cdot 1_{A})(h_1\cdot b)\#h_2g\\
		&=&a(1_H\cdot( S_R^{-1}(z)\cdot 1_{A})(h_1\cdot b))\#h_2g\\
		&=&a(1_1\cdot S_R^{-1}(z)\cdot 1_{A})(1_2\cdot h_1\cdot b)\#h_2g\\
		&=&a(1_1\cdot S_R^{-1}(z)\cdot 1_{A})(1_2\cdot  1_{A})(1_3h_1\cdot b)\#h_2g\\
		&=&a(1_1\cdot S_R^{-1}(z)\cdot 1_{A})(1_2 h_1\cdot b)\#h_2g\\
		&\Ref{noHRcola}&a(1_1 S_R^{-1}(z)\cdot 1_{A})(1_2 h_1\cdot b)\#h_2g\\
		&\overset{\eqref{p13-1}}{=}&a(1_1\cdot 1_{A})(1_2zh_1\cdot b)\#h_2g\\
		&=&a(1_1\cdot 1_{A})(1_2\cdot 1_{A})(1_3zh_1\cdot b)\#h_2g\\
		&=&a(1_1\cdot 1_{A})(1_2\cdot zh_1\cdot b)\#h_2g\\
		&=&a(1_H\cdot 1_{A}(zh_1\cdot b))\#h_2g\\
		&=&a(z1_1h_1\cdot b)\#1_2h_2g\\
		&\Ref{p3-1}&a(z_1h_1\cdot b)\#z_2h_2g\\
		&=&\tilde{\mu}(a,zh,b,g).
	\end{array}
	$$

\vu

	The associativity:
	$$
	\begin{array}{rcl}
		((a\# h)(b\# g))(c\# k)&=&(a(h_1\cdot b)\# h_2g)(c\# k)\\
		&=&a(h_1\cdot b)(h_2g_1\cdot c)\# h_3g_2k\\
		&=&a(h_1\cdot b)(h_2\cdot 1_{A})(h_3g_1\cdot c)\# h_4g_2k\\
		&=&a(h_1\cdot b)(h_2\cdot(g_1\cdot c))\# h_3g_2k\\
		&=&a(h_1\cdot b(g_1\cdot c))\# h_2g_2k\\
		&=&(a\#h)(b(g_1\cdot c)\# g_2k)\\
		&=&(a\# h)((b\# g)(c\# k)).
	\end{array}
	$$

\vu

The left unit:
	$$
	\begin{array}{rcl}
		( 1_{A}\#1_H)(a\# h)&=& 1_{A}(1_1\cdot a)\#1_2h\\
		&=& (1_1\cdot a)\#1_2h\\
		&\Ref{p2}& S_R^{-1}(S(1_1)\cdot a)\#1_2h \\
		&=& a \triangleleft S(1_1)\# 1_2h \\
		&=& a \# S(1_1)1_2 h \\
		&=& a \# \varepsilon_L(1_H) h \\
		&=& a \# h.
	\end{array}
	$$
\end{proof}
It follows from the above theorem that
$$A \underline{\#} H= (A\# H)( 1_{A}\#1_H)$$ is an algebra with $1_A\#1_H$ as its unit.
This algebra is called the \emph{partial smash product} of $A$ by $H$.

\vu

The following example illustrates that, in general, $1_A\# 1_H$ is not a unit of $A\# H$.

\begin{ex}
Let $G$ be a finite groupoid which is not a group and $G_e=\{g\in G\ \mid \ d(g)=e=r(g)\}$ the isotropy group
associated to $e$, for some $e\in G_0$. It is easy to see that $\Bbbk$ is a partial $\Bbbk G$-module algebra via
$$
\begin{array}{rccl}
\cdot \colon & \Bbbk G \otimes \Bbbk & \to & \Bbbk \\
& g \otimes 1_{\Bbbk} & \mapsto & \delta_{g,G_e}
\end{array}
$$
where $\delta_{g,G_e} = 1_\Bbbk$ if $g \in G_e$ and $0$ otherwise.
In this case, $1_{\Bbbk} \# 1_{\Bbbk G}$ is not a right unit for the smash product $\Bbbk \# \Bbbk G$.
Indeed, since $G$ is not a group, there exists an element $x$ in $G\setminus G_e$.
Thus, $x \cdot 1_{\Bbbk} = 0$ and, consequently, $(1_{\Bbbk} \# x)(1_{\Bbbk} \# 1_{\Bbbk G})=0$.\qed
\end{ex}

\vu

Actually, we have the following.

\begin{prop}
Let $A$ be a partial $H$-module algebra. Then, $1_A \# 1_H$ is a unit in $A \# H$ if and only if $A$ is an $H$-module algebra.
\end{prop}

\begin{proof}
Suppose $1_A \# 1_H$ a unit in $A \# H$. Then $a\# h = a(h_1 \cdot 1_A) \# h_2$ and, applying $\El$ on the second element of
each term of this equality,
we have $$a \triangleleft \El(h) = a (h_1 \cdot 1_A) \triangleleft \El(h_2).$$ Hence,
$$
\begin{array}{rcl}
a(\El(h) \cdot 1_A) &\Ref{acaodireita}& a\triangleleft \El(h) \\
&=& a(h_1 \cdot 1_A) \triangleleft \El(h_2) \\
&=& a(h_1 \cdot 1_A)(\El(h_2) \cdot 1_A) \\
&=& a(h_1 \cdot 1_A)(h_2 S(h_3) \cdot 1_A) \\
&\overset{(9)}{=}& a(1_1 h \cdot 1_A)(1_2 \cdot 1_A) \\
&\Ref{lemaadicional}& a(h \cdot 1_A)(1_H \cdot 1_A)\\
&=& a(h \cdot 1_A)
\end{array}
$$
and, taking $a=1_A$ we have $h \cdot 1_A = \El(h) \cdot 1_A$. By Lemma \ref{condglobal}, $A$ is an $H$-module algebra.
The converse is straightforward and standard.
\end{proof}

\vd

%
%

\section{A Morita Context}\label{morita}

In the setting of partial actions of Hopf algebras with an invertible antipode there exits a Morita context relating
the partial smash product $A \underline{\#} H$ and the (global) smash product $B\# H$, where $B$ denotes
a globalization of $A$ such that the image of $A$ inside $B$ is an ideal of $B$ (see \cite{AB}). In this section we
extend this result to the setting of partial actions of weak Hopf algebras.

\vu

First, recall the definition of a Morita context.

\begin{dfn}
	Let $A$ and $B$ be unital rings. A Morita context for $A$ and $B$ is a sixtuple $(A, B, M, N, (,), [,])$ where
$M$ is a $(A,B)$-bimodule, $N$ is a $(B,A)$-bimodule, and $(, )\colon M\otimes_{B} N  \to A$ and $[, ]\colon N\otimes_{A} M  \to B$ are homomorphisms of $(A,A)$-bimodules and
$(B,B)$-bimodules, respectively,  such that

\begin{itemize}

\item[(i)]	$(m, n) m' = m [n, m']$

\vu

\item[(ii)] $[n, m] n' = n (m, n')$,
\end{itemize}
for all $m,m'\in M$ $n,n'\in N$.
\end{dfn}

We will first construct a non unitary monomorphism of algebras from $A\# H$ into $B\# H$.
For this we need the following lemma.

\begin{lema}\label{p=g}
Let $A$ be a partial $H$-module algebra and $(B,\theta)$ a globalization of $A$. Then,
$S_R^{-1}(h) \triangleright \theta(a) = S_R^{-1}(h) \cdot \theta(a)$, for all $h \in H_L$.
\end{lema}

\begin{proof}
	Note that $S_R^{-1}(h) \cdot \theta(a) = \theta(1_A)(S_R^{-1}(h) \triangleright \theta(a))$. But $S_R^{-1}(h)$ lies in $H_R$, thus, by Proposition \ref{saiprafora}(i), we have $\theta(1_A)(S_R^{-1}(h) \triangleright \theta(a)) =
S_R^{-1}(h) \triangleright(\theta(1_A) \theta(a)) = S_R^{-1}(h) \triangleright \theta(a)$.
\end{proof}

\vu

\begin{prop}\label{monofrompartialtoglobal}
	Let $(B, \theta)$ be a globalization of the partial $H$-module algebra $A$. Then, there exists a
non unitary algebra monomorphism $\Psi$ from $A\# H$ to $B\# H$.
\end{prop}

\begin{proof}
	Define
	$$\begin{array}{rl}
		\tilde{\Psi}\colon A \times H &\to B\otimes_{H_L} H\\
		(a, h)&\mapsto \theta(a)\otimes h.
	\end{array}$$

	For all  $a\in A, h\in H$ and $z\in H_L$, we have
	$$
	\begin{array}{rcl}
		\tilde{\Psi}(a, zh)&=& \theta(a)\otimes z h\\
		&=&\theta(a)\triangleleft z\otimes h\\
		&=&S_R^{-1}(z)\triangleright\theta(a)\otimes h\\
		&\overset{\eqref{p=g}}{=}& S_R^{-1}(z)\cdot\theta(a)\otimes h\\
		&=&\theta(S_R^{-1}(z)\cdot a)\otimes h\\
		&=&\theta(a\triangleleft z)\otimes h\\
		&=&\tilde{\Psi}(a\triangleleft z, h),
	\end{array}
	$$
which shows that $\tilde{\Psi}$ is $H_L$-balanced. Thus, there exists a $\Bbbk$-linear
map $\Psi$ from $A\otimes_{H_L} H$ to   $B\otimes_{H_L} H$
defined by $\Psi(a\otimes h) = \theta(a)\otimes h$.

\vu

It follows from the injectivity of $\theta$ and  similar calculations that the
$\Bbbk$-linear map $\Psi'\colon \theta(A) \otimes_{H_L} H \to A \otimes_{H_L} H$
given by $\Psi'(\theta(a)\otimes h) = a\otimes h$ is well defined. Furthermore, $\Psi$
is a monomorphism because $\Psi'\circ\Psi = I_{A\otimes_{H_L} H}$.

\vu

It remains to check that $\Psi\colon A\# H\to B\# H$ is multiplicative. In fact, for all $a,b\in A$ and $h,g\in H$,
we have
	$$
		\begin{array}{rcl}
			\Psi((a\# h)(b\# g))&=&\Psi(a(h_1\cdot b)\#h_2 g)\\
			&=&\theta(a (h_1\cdot b))\#h_2 g\\
			&=&\theta(a) \theta(h_1\cdot b)\#h_2 g\\
			&=&\theta(a) (h_1\cdot \theta(b))\#h_2 g\\
			&=&\theta(a) (h_1\triangleright \theta(b))\#h_2 g\\
			&=&(\theta(a)\# h)(\theta(b)\# g)\\
			&=&\Psi(a\# h)\Psi(b\# g).
		\end{array}
	$$
\end{proof}

In the sequel, we will construct the bimodules which will define a Morita context for $A \underline{\#} H$ and $B\# H$.
For this construction we will suppose that $\theta(A)$ is an ideal of B and the antipode $S$ of $H$ is invertible. This assumption
on $S$ is necessary because in this construction we will need to make use of Proposition \ref{2.25}(ii).

\vu

Now let $M=\Psi(A\# H)$ and $N$ be the vector space generated by the elements of the form
$(h_1\triangleright\theta(a))\# h_2$, for all $a\in A$ and $h\in H$.

\begin{prop}
	With the above notations and assumptions, $M$ is a right $B\# H$-module and $N$ is
a left $B\# H$-module, via the multiplication of $B\# H$.
\end{prop}

\begin{proof} Let $\theta(a)\# h\in M$ and $k\triangleright\theta(b)\# g\in B\# H$, so
$$
	\begin{array}{rcl}
		(\theta(a)\# h)(k\triangleright\theta(b)\# g)&=&\theta(a)(h_1k\triangleright\theta(b))\# h_2g\\
		&=&\theta(a)(h_1 k\cdot \theta(b))\# h_2 g\\
		&=&\theta(a(h_1 k\cdot b))\# h_2 g
	\end{array}
$$
	that lies in $M$.

\vd
	
	Let $k\triangleright\theta(a)\# h \in B\# H$ and $g_1\triangleright\theta(b)\# g_2 \in N$. Then we have
	$$
		\begin{array}{rcl}
			(k\triangleright\theta(a)\# h)(g_1\triangleright\theta(b)\# g_2)&=&(k\triangleright\theta(a)) (h_1g_1\triangleright\theta(b))\# h_2g_2\\
			&\Ref{2.25}&h_2g_2\triangleright[(S^{-1}(h_1g_1) k\triangleright\theta(a))\theta(b)]\# h_3g_3
		\end{array}
	$$
	that lies in $N$ because $\theta(A)$ is an ideal of $B=H\triangleright\theta(A)$.

\vu

Now, the assertion follows from the associativity of $B\# H$.
\end{proof}

\begin{prop}
Keeping the same notations and  assumptions as above, $M$ is a left $A\underline{\#}H$-module  and $N$ is a right
$A \underline{\#} H$-module via the actions
$$
\begin{array}{cccc}
	\blacktriangleright\colon & A \underline{\#} H \otimes M & \to & M\\
	& a \underline{\#} h \otimes m & \mapsto & \Psi(a \underline{\#} h ) m
\end{array}
$$
and
$$
\begin{array}{cccc}
	\blacktriangleleft\colon & N \otimes A \underline{\#} H & \to & N\\
	& n \otimes a \underline{\#} h & \mapsto & n \Psi(a \underline{\#} h)
\end{array}
$$ respectively, where $\Psi$ is the non unitary monomorphism defined in 7.3. Moreover, $M$ and $N$ are bimodules.
\end{prop}

\begin{proof}
	We need only to ensure that $\blacktriangleleft$ is well defined. The well definition
of $\blacktriangleright$ as well as the other assertions follow from the fact that $A \underline{\#} H$ is a subalgebra of $A\# H$ and $\Psi$
is  multiplicative.

\vu
	
	Given $h_1\triangleright\theta(a)\# h_2\in N$ and $a' \underline{\#} g \in A \underline{\#} H$, we have
	$$
		\begin{array}{rcl}
			(h_1\triangleright\theta(a)\# h_2)\blacktriangleleft(a'\underline{\#} g) &=& (h_1\triangleright\theta(a)\# h_2)(\Psi(a'\underline{\#} g))\\
			&=& (h_1\triangleright\theta(a)\# h_2)(\theta(a'(g_1\cdot 1_A))\# g_2)\\
			&=& (h_1\triangleright\theta(a)\# h_2)(\theta(a')(g_1\cdot \theta(1_A)\# g_2))\\
			&=& (h_1\triangleright\theta(a)\# h_2)(\theta(a')(g_1\triangleright \theta(1_A)\# g_2))\\
			&=& (h_1\triangleright\theta(a)) (h_2\triangleright\theta(a'))(h_3g_1\triangleright \theta(1_A))\# h_4 g_2\\
			&=& (h_1\triangleright\theta(a a'))(h_2 g_1\triangleright \theta(1_A))\# h_3 g_2\\
			&\Ref{2.25}& h_3 g_2\triangleright [(S^{-1}(h_2 g_1) h_1\triangleright\theta(a a'))\theta(1_A)]\# h_4 g_3
		\end{array}
	$$
	that lies in $N$, because $\theta(A)$ is an ideal of $B$, which ensures that $\blacktriangleleft$ is also well defined.
\end{proof}

Now, we consider the maps $[,]\colon N\otimes_{A \underline{\#} H} M \to B\# H$ and $(,)\colon M\otimes_{B\# H} N\to \Psi(A \underline{\#} H)\simeq A \underline{\#} H$
given by the multiplication of $B\# H$. Both such maps are well defined because $M,N\subseteq B\# H$.

\vu

\begin{teo}
	$(A \underline{\#} H, B\# H, M, N, (,), [,] )$ is a Morita context.
Moreover, the maps $[,]$ and $(,)$ are both surjective. In particular, if $B$ also has an identity element, then $A \underline{\#} H$ and  $B\# H$ are Morita equivalent.
\end{teo}

\begin{proof}
	The main assertion follows from Propositions 7.4 and 7.5, and from the associativity of the multiplication of $B\# H$.

\vu
	
	For the surjectivity of $(,)$ and $[,]$ it is enough to show that $M N = \Psi(A \underline{\#} H)$ and $N M = B\# H$.
	
	In fact, clearly $\Psi(A \underline{\#} H)\subseteq M N$. Conversely, given $g_1\triangleright\theta(b)\# g_2\in N$ and
$\theta(a)\# h\in M$ we have
	$$
		\begin{array}{rcl}
			(\theta(a)\#h)(g_1\triangleright\theta(b)\# g_2)&=& \theta(a)(h_1 g_1\triangleright\theta(b))\# h_2 g_2\\
			&=& \theta(a)(h_1 g_1\cdot \theta(b))\# h_2 g_2\\
			&=& \theta(a(h_1 g_1\cdot b))\# h_2 g_2\\
			&=& \theta(a(h_1 g_1\cdot b1_A))\# h_2 g_2\\
			&=& \theta(a(h_1 g_1\cdot b)(h_2g_2\cdot 1_A))\# h_3 g_3\\
			&=& \Psi(a(h_1 g_1\cdot b) \underline{\#} h_2g_2)
		\end{array}
	$$
	which lies in $\Psi(A \underline{\#} H)$. Hence, $M N = \Psi(A \underline{\#} H)$.

\vu
	
	Clearly, we have that $N M \subseteq B\# H$. To prove that $B\# H \subseteq N M$ it is enough to check that the equality
	$$(h_1\triangleright\theta(a)\# h_2)(\theta(1_A)\#S(h_3)g) = h\triangleright\theta(a)\# g$$
holds for all $a,b\in A$ and $h,g\in H$. Indeed,
$$
	\begin{array}{rcl}
		(h_1\triangleright\theta(a)\# h_2)(\theta(1_A)\#S(h_3)g)&=& (h_1\triangleright\theta(a))(h_2\triangleright \theta(1_A))\#h_3 S(h_4)g\\
		&=& h_1\triangleright\theta(a)\# h_2 S(h_3) g\\
		&\Ref{p4-1}& 1_1 h\triangleright\theta(a)\# 1_2 g\\
		&\Ref{noHRcola}& 1_1\triangleright h\triangleright\theta(a)\# 1_2 g\\
		&=& (h\triangleright\theta(a))\triangleleft S(1_1)\# 1_2 g\\
		&=& h\triangleright\theta(a) \# S(1_1) 1_2 g\\
		&=& h\triangleright\theta(a) \# \varepsilon_R(1_H) g\\
		&=& h\triangleright\theta(a)\# g.
	\end{array}
$$
Therefore, $B\# H=N M$. The last assertion follows from \cite{Rowen}*{Theorems 4.1.4 and 4.1.17}
\end{proof}
\section{Acknowledgments}
The authors would like to thank Daiana A. Fl\^ores for her manuscript concerning to infinite dimensional weak Hopf algebras,
which makes part of a mini course given by herself at the ``Colloquium on Algebra and Representations - Quantum 2014''
and can be found at http://www.famaf.unc.edu.ar/$\sim$plavnik/quantum2014-en/index-en.html. The authors would also like to
thank the referee, whose comments and suggestions were very useful to improve this last version of the manuscript.

\end{document}